\documentclass[a4paper,12pt, reqno]{amsart}

\usepackage{amsmath, amsthm, amscd, amsfonts, amssymb, graphicx, color}
\usepackage[bookmarksnumbered, colorlinks, plainpages]{hyperref}

\usepackage{enumitem}
\usepackage{url}
\usepackage{multirow}
\usepackage{graphicx}
\usepackage{subfigure}
\usepackage[pass]{geometry}
\usepackage{cite}
\usepackage{cancel}
\usepackage{color}

\usepackage{lscape}

\usepackage{tikz}
\usetikzlibrary{calc,shapes,decorations.pathreplacing,matrix}
\usetikzlibrary{patterns}

\tikzset{square matrix/.style={
    matrix of nodes,
    column sep=-\pgflinewidth, row sep=-\pgflinewidth,
    nodes={draw,
      minimum height=4.5pt,
      anchor=center,
      text width=4.5pt,
      align=center,
      inner sep=0pt
    },
  },
  square matrix/.default=1.2cm
}

\newtheorem{Theorem}{Theorem}[section]
\newtheorem{Definition}[Theorem]{Definition}

\newtheorem{Lemma}[Theorem]{Lemma}

\newtheorem{Remark}[Theorem]{Remark}
\newtheorem{Example}[Theorem]{Example}

\begin{document}

\title{Domination for latin square graphs}

\author{Behnaz Pahlavsay}
\address{Behnaz Pahlavsay, Department of Mathematics, Hokkaido University, Kita 10, Nishi 8, Kita-Ku, Sapporo 060-0810, Japan.}
\email{pahlavsayb@gmail.com}
\author{Elisa Palezzato}
\address{Elisa Palezzato, Department of Mathematics, Hokkaido University, Kita 10, Nishi 8, Kita-Ku, Sapporo 060-0810, Japan.}
\email{palezzato@math.sci.hokudai.ac.jp}
\author{Michele Torielli}
\address{Michele Torielli, Department of Mathematics, GI-CoRE GSB, Hokkaido University, Kita 10, Nishi 8, Kita-Ku, Sapporo 060-0810, Japan.}
\email{torielli@math.sci.hokudai.ac.jp}

\date{\today}

\begin{abstract}
In combinatorics, a \emph{latin square} is a $n\times n$ matrix filled with $n$ different symbols, each occurring exactly once in each row and exactly once in each column. Associated to each latin square, we can define a simple graph called a \emph{latin square graph}. In this article, we compute lower and upper bounds for the domination number and the $k$-tuple total domination numbers of such graphs. Moreover, we describe a formula for the $2$-tuple total domination number.
%
%
\end{abstract}
\maketitle

\section{Introduction}

Domination is well-studied in graph theory and the literature on this subject has been surveyed and detailed in the two books by Haynes, Hedetniemi, and Slater~\cite{HHS5, HHS6}. Throughout this paper, we use standard notation for graphs, see for example \cite{bondy2008graph}.

\begin{Definition} Let $G=(V_G,E_G)$ be a simple graph. A set $S \subseteq V_G$ is called a \emph{dominating set} if every vertex $v \in V_G\setminus S$ has at least one neighbour in $S$, i.e., $|N_G(v)\cap S| \geq 1$, where $N_G(v)$ is the open neighbourhood of $v$.  The \emph{domination number}, which we denote by $\gamma(G)$, is the minimum cardinality of a dominating set of $G$.
\end{Definition}

Similarly to \cite{KPPT2008}, the notion of domination has a central role in this paper. Among its many variations, we are also interested in $k$-tuple total domination, which was introduced by Henning and Kazemi \cite{HK8} as a generalization of \cite{HH3}, and also recently studied in \cite{PPT2018}.   

\begin{Definition} Let $G=(V_G,E_G)$ be a simple graph and $k \geq 1$. A set $S \subseteq V_G$ is called a $k$-\emph{tuple total dominating set} ($k$TDS) if every vertex $v \in V_G$ has at least $k$ neighbours in $S$, i.e., $|N_G(v)\cap S| \geq k$.  The \emph{$k$-tuple total domination number}, which we denote by $\gamma_{\times k,t}(G)$, is the minimum cardinality of a $k$TDS of $G$.  
We use min-$k$TDS to refer to  $k$TDSs of minimum size.
\end{Definition}

Since, by definition, every $k$TDS is a dominating set, we have that for all $k\ge1$
\begin{equation}\label{eq:gammasmallgammatot}
\gamma(G)\le\gamma_{\times k,t}(G).
\end{equation}

An immediate necessary condition for a graph to have a $k$-tuple total dominating set is that every vertex must have at least $k$ neighbours.  For example, for $k \geq 1$, a $k$-regular graph $G=(V_G,E_G)$ has only one $k$-tuple total dominating set, namely $V_G$ itself. Moreover, notice that a $k$-tuple total dominating set has at least $k+1$ elements.

\section{Latin squares}
Latin squares have been around for several centuries. In recent years, together with their associated graph, they have been intensively studied because of their connections with other area of mathematics and their practical applications. We refer to \cite{CombinDes} for an introduction on latin squares.

\begin{Definition}
A \emph{latin square of order $n\ge1$} is a $n\times n$ matrix containing $n$ symbols such that each row and each column contains exactly one copy of each symbol.
\end{Definition} 

\begin{Definition}
In a latin square $L$ of order $n$, if, for some $1\le l \le n$, the $l^2$ cells defined by $l$ rows and $l$ columns form a latin square of order $l$, 
it is called a \emph{latin subsquare} of $L$. 
\end{Definition}
Unless otherwise specified, in this paper we use $[n] = \{1,2,\dots ,n\}$ as the symbol set and also to index the rows and columns of a latin square.

It is clear that, if we permute in any way the rows, or the columns, or the symbols of a latin square, the result is still a latin square. 

Let $L$ be a latin square of order $n$ with cells $\{(r,c)~|~r,c\in[n]\}$, then each cell contains a symbol from an alphabet of size $n$, and no row or column of $L$ contains a repeated
symbol. Hence, given a cell $(r , c)$ containing the symbol $s = L_{r ,c}$, we can represent it by the triple $(r , c , s)$.

\begin{Definition} For a latin square $L$, we define 
$$E(L)=\{(r , c , s)~|~r,c,s\in[n]\text{ and }s = L_{r ,c}\}$$
to be the set of \emph{entries} of $L$.
\end{Definition}

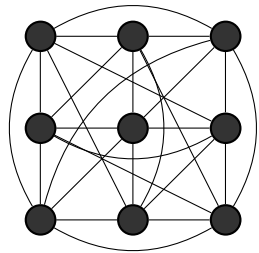
\begin{figure}[htp]
\centering
\begin{minipage}{0.4\textwidth}
\centering
\begin{tikzpicture}
\matrix[square matrix,nodes={draw,
      minimum height=11pt,
      anchor=center,
      text width=11pt,
      align=center,
      inner sep=0pt
    },]{
 1 & 2 & 3 \\
 2 & 3 & 1 \\
 3 & 1 & 2 \\
};
\end{tikzpicture}
\end{minipage}
\begin{minipage}{0.4\textwidth}
\centering
\begin{tikzpicture}
\matrix[nodes={draw, thick, fill=black!80, circle},row sep=0.8cm,column sep=0.8cm] {
  \node(11){}; &
  \node(12){}; &
  \node(13){};\\
  \node(21){}; &
  \node(22){}; &
  \node(23){};\\
  \node(31){}; &
  \node(32){}; &
  \node(33){}; \\
};
\draw (11) to (12) to (13) ; \draw (11) to[bend left] (13); 
\draw (21) to (22) to (23) ; \draw (21) to[bend right] (23); 
\draw (31) to (32) to (33) ; \draw (31) to[bend right] (33); 
\draw (11) to (21) to (31) to[bend left] (11);  \draw (12) to (22) to (32) to[bend right] (12);  \draw (13) to (23) to (33) to[bend right] (13);  
\draw (11) to (23) to (32) to (11) ;
\draw (12) to (21) to (33) to (12) ;
\draw (13) to (22) to (31) to[bend left] (13) ;
\end{tikzpicture}
\end{minipage}
\caption{A latin square of order $3$ and its associated latin square graph.}\label{fig:latinsqu3x3}
\end{figure}

\begin{Definition}
Two latin squares $L$ and $L'$ (using the same symbol set) are \emph{isotopic} if there is a triple $(\sigma,\tau,\delta)$, where $\sigma$ is a row permutation, $\tau$ a column permutation, and $\delta$ a 
symbol permutation, carrying $L$ to $L'$. This means that if $(r , c , s)$ is an entry of $L$, then $(\sigma(r), \tau(c), \delta(s))$ is an entry of $L'$. The triple $(\sigma,\tau,\delta)$ is called an \emph{isotopy}.
\end{Definition}

In the theory of latin squares, the notion of partial transversal plays an important role. 

\begin{Definition}
Let $L$ be a latin square of order $n$.
A \emph{partial transversal} is a subset of $E(L)$ such that no two entries share the same row, column or symbol. 
We say that a partial transversal is \emph{maximal} if it is not properly contained in any other partial transversal. 
A \emph{transversal} is a partial transversal of cardinality $n$, i.e. it is a set of entries which includes exactly one entry from each row, column and symbol.
\end{Definition}

%
\begin{Example} 
If we consider $L$ the latin square of Figure~\ref{fig:latinsqu3x3}, we can easily construct a transversal by considering the elements on the main diagonal, i.e. $\{(1,1,1),(2,2,3),(3,3,2)\}$.
\end{Example}
%
%

The idea of constructing a simple graph from a latin square was introduced by Bose in \cite{bose1963} as example of strongly regular graphs. See \cite[Section 10.4]{GR01} for further discussion.
\begin{Definition} The \emph{latin square graph} of a latin square $L$ is the simple graph $\Gamma(L)$ with vertex set $E(L)$
and an edge between two distinct entries whenever they share a row, a column or a symbol. 
Accordingly, each edge of $\Gamma(L)$ is called, respectively, a row edge, a column edge or a symbol edge.
\end{Definition} 

\begin{Remark} Latin square graphs are invariant under isotopy, i.e. if two latin squares are isotopic, then their associated graphs are isomorphic.
\end{Remark}


It is trivial that $\Gamma(L)$ is the complete graph on $n^2$ vertices if and only if $n = 1,2$. 
A latin square graph $\Gamma(L)$ is a $3(n-1)$-regular graph, and any two different vertices $(r, c, s)$ and $(r, c', s')$ have $n$ neighbors in common, 
i.e. $n-2$ vertices in the row $r$ and the two vertices in the columns $c$ and $c'$ with symbols $s'$ and $s$, respectively. 
Similarly, any two different vertices $(r, c, s)$ and $(r', c, s')$ have $n$ neighbors in common. 
Moreover, it can be easily seen that any two distinct vertices $(r, c, s)$ and $(r', c', s)$ have also $n$ neighbours in common.

As noted in \cite{BMSW}, any maximal partial transversal of $L$ corresponds to a dominating set in $\Gamma(L)$. Notice that, however, not all dominating sets correspond to a maximal partial transversal. For example, if we consider $L$ the latin square of Figure~\ref{fig:latinsqu3x3}, we can easily construct a dominating set of $\Gamma(L)$ that is not a partial transversal by considering the set $\{(1,1,1),(2,1,2),(3,1,3)\}$.

\section{$k$-tuple total dominating set}

In \cite{KP} and \cite{EMW19}, the authors investigated the relationship between domination in latin square graphs and transversal in latin squares, motivating the study of various types of domination for such graphs. 

Since every latin square graph $\Gamma(L)$ is a $3(n-1)$-regular graph, when studying min-$k$TDS we should consider $k\le 3(n-1)$. This fact together with the definition of $\gamma_{\times k,t}(G)$ give us the following. 

\begin{Lemma} 
Let $L$ be a latin square of order $n$. Then $$\gamma_{\times 3(n-1),t}(\Gamma(L))=n^2.$$
\end{Lemma}

If we consider a latin square $L$ of small order, we can easily compute $\gamma_{\times k,t}(\Gamma(L))$.
\begin{Lemma}\label{lemma:n2case}
Let $L$ be a latin square of order $2$. Then $\gamma_{\times 1,t}(\Gamma(L))=2$ and $\gamma_{\times 2,t}(\Gamma(L))=3$.
\end{Lemma}
\begin{proof} This is a direct consequence of the fact that every $k$-tuple total dominating set has at least $k+1$ elements and the fact that any choice of two vertices in 
$\Gamma(L)$ gives a $1$TDS and any choice of $3$ vertices in $\Gamma(L)$ gives a $2$TDS.
\end{proof}

\begin{Lemma}\label{lemma:n3case} 
Let $L$ be a latin square of order $3$. Then $\gamma_{\times 1,t}(\Gamma(L))=2$ and $\gamma_{\times 2,t}(\Gamma(L))=3$.
\end{Lemma}
\begin{proof} We first study $\gamma_{\times 1,t}(\Gamma(L))$. Assume $S=\{(1,1,s_1),(1,2,s_2)\}$. Then $S$ is a $1$TDS, in fact every vertex $v$ of $\Gamma(L)$ has at least one neighbour in $S$. Hence $\gamma_{\times 1,t}(\Gamma(L))\le |S|=2$. Since every $1$-tuple total dominating set has at least $2$ elements, we have that $\gamma_{\times 1,t}(\Gamma(L))=2$.

We now study $\gamma_{\times 2,t}(\Gamma(L))$. Assume $S=\{(r,c,1)~|~1\le r,c\le 3\}$. Then $S$ is a $2$TDS, in fact every vertex $v$ of $\Gamma(L)$ has at least $2$ neighbours in $S$. Hence $\gamma_{\times 2,t}(\Gamma(L))\le|S|=3$. Since every $2$-tuple total dominating set has at least $3$ elements, we have that $\gamma_{\times 2,t}(\Gamma(L))=3$.

\end{proof}

We can now describe a general upper bound for $\gamma_{\times k,t}(\Gamma(L))$.
\begin{Theorem}\label{theo:ktdsupperbound}
Let $L$ be a latin square of order $n\ge4$.
Then 
$$
\gamma_{\times k,t}(\Gamma(L))\le
\begin{cases}
  n-1 & \text{ if $k=1$} \\
  an  & \text{ if $k=2a$ and $1\le a\le n$}  \\
  an+n-a & \text{ if $k=2a+1$ and $1\le a\le n-2$.} 
  \end{cases}
$$
\end{Theorem}
\begin{proof}
Consider $k=1$ and $S=\{(1,1,s_1),(1,2,s_2),\dots,(1,n-1,s_{n-1})\}$, i.e. $S$ consists of the $n-1$ vertices of $\Gamma(L)$ corresponding to the first $n-1$ entries in the first row of $L$. Now if $v\in S$, then it has $n-2$
neighbours in $S$. If $v=(1,n,s_n)$, then it has $n-1$ neighbours in $S$. If $v=(r,c,s)$ with $2\le r \le n$ and $1\le c\le n-1$, then $v$ has at least $1$ (possibly two)  neighbour in $S$, i.e. $(1,c,s_c)$.
Finally, if $v=(r,n,s)$ with $2\le r\le n$, then $v$ has exactly $1$ neighbour in $S$, i.e. $(1,c_i,s_i)$ with $s_i=s$. This shows that $S$ is a $1$TDS. Hence $\gamma_{\times 1,t}(\Gamma(L))\le |S|=n-1$.

Consider $k=2a$ and $S=\{(r,c,s)~|~1\le r,c\le n \textit{ and } 1\le s \le a\}$, i.e. $S$ consists of all the vertices of $\Gamma(L)$ corresponding to entries of $L$ where the symbols $1,2,\dots, a$ appear. 
Notice that in each row or column of $L$, $S$ has exactly $a$ vertices. This shows that if $v$ is not in $S$, then it has exactly $2a$ neighbours in $S$. On the other hand, each $v\in S$ has $n-1+2(a-1)=n+2a-3\ge2a$
neighbour in $S$. This implies that $S$ is a $2a$TDS. Hence $\gamma_{\times 2a,t}(\Gamma(L))\le |S|=an$.

Consider $k=2a+1$ and $S=\{(r,c,s)~|~1\le r,c\le n \textit{ and } 1\le s \le a\}\cup\{(1,c,s)~|~s\ge a+1\}$, i.e. $S$ consists of all the vertices of $\Gamma(L)$ corresponding to entries of $L$ where the symbols $1,2,\dots, a$ appear together
with the remaining $n-a$ entries of the first row of $L$.  Consider $v=(r,c,s)$ a vertex of $\Gamma(L)$ not in $S$. Then there are $a$ elements of $S$ in the same row of $v$, $a$ elements of $S$ in the same column of $v$ 
and $1$ element $(1,c',s)$ in $S$ with $c'\ne c$. Hence, $v$ has exactly $2a+1$ neighbours in $S$. If $v=(r,c,s)\in S$ with $1\le s \le a$, then it has at least $2(a-1)+n-1=n+2a-3\ge2a+1$ neighbours in $S$.
If $v=(1,c,s)\in S$ with $s\ge a+1$, then it has exactly $n-1+a\ge 2a+1$. This implies that $S$ is a $(2a+1)$TDS. Hence $\gamma_{\times 2a+1,t}(\Gamma(L))\le |S|=an+n-a$.
\end{proof}


Using the previous results, we can now compute $\gamma_{\times 2,t}(\Gamma(L))$ for any latin square $L$.
\begin{Theorem} Let $L$ be a latin square of order $n\ge3$. Then $$\gamma_{\times 2,t}(\Gamma(L))=n.$$
\end{Theorem}
\begin{proof} If $n=3$, by Lemma~\ref{lemma:n3case}, $\gamma_{\times 2,t}(\Gamma(L))=3=n$. 

Assume $n\ge4$. By Theorem~\ref{theo:ktdsupperbound}, $\gamma_{\times 2,t}(\Gamma(L))\le n$. Suppose, seeking a contradiction, that $\gamma_{\times 2,t}(\Gamma(L))\le n-1$ and hence, there exists $S$ a $2$TDS with $|S|=n-1$.
Without loss of generalities, we can assume that 
$S$ has no vertices corresponding to entries from the last row or column of $L$ 
or corresponding to entries with the symbol $n$. 

We claim that if $S$ has no vertices corresponding to entries from $k$ rows (respectively $k$ columns) of $L$, for some $k\ge1$, then $S$ has no vertices corresponding to entries from at least $k+1$ columns (respectively $k+1$ rows) of $L$. To prove the claim we assume that $S$ has no vertices corresponding to entries from $k$ rows of $L$. Since $L$ is a latin square of order $n$, in each one of these $k$ rows there is one entry with symbol $n$, and all these $k$ entries are in different columns. If any of these $k$ entries with symbol $n$ is in the last column of $L$, then the corresponding vertex would have no neighbour in $S$, and hence, $S$ would not be a $2$TDS. We can then assume that none of these $k$ entries is in the last column of $L$. 
Furthermore, because $S$ is a $2$TDS, each of the vertices corresponding to these $k$ entries with symbol $n$ has at least $2$ neighbours in $S$. This implies that in $L$ there are at least $k$ columns with at least $2$ entries corresponding to vertices in $S$.
Since $|S|=n-1$, this implies that in $L$ there are at least $k+1$ columns (one is the last one) such that none of the entries in these columns correspond to vertices in $S$. The same type of argument works for the case that $S$ has no vertices corresponding to entries from $k$ columns of $L$. This proves the claim.

Using the claim, we have that $S$ has no vertices corresponding to entries from $k$ rows (and $k$ columns) of $L$ for all $k\ge1$, but this is impossible.
\end{proof}


We can also describe a lower bound for $\gamma_{\times 1,t}(\Gamma(L))$.
\begin{Theorem}\label{theo:gamma1lowerbound4n7}
Let $L$ be a latin square of order $n\ge 2$. Then
$$\gamma_{\times 1,t}(\Gamma(L)) \geq \frac{4n -2}{7}.$$
\end{Theorem}
\begin{proof}
If $n=2,3$ the result follows from Lemma~\ref{lemma:n2case} and Lemma~\ref{lemma:n3case}.
Assume $n\geq 4$, and consider $S$ a 1TDS with $|S|=\gamma_{\times 1,t}(\Gamma(L))$.
For all $i=1,\dots,n$, define $\mathfrak{r}_i=|\{(i,c,s)\in S ~|~ 1\leq c,s \leq n\}|$, $\mathfrak{c}_i=|\{(r,i,s)\in S ~|~ 1\leq r,s \leq n\}|$ and $\mathfrak{s}_i=|\{(r,c,i)\in S ~|~ 1\leq r,c \leq n\}|$.
Moreover, define 
$$\mathfrak{r}=\sum_{\substack{i=1 \\ \mathfrak{r}_i\neq 0}}^n (\mathfrak{r}_i-1),~
\mathfrak{c}=\sum_{\substack{i=1 \\ \mathfrak{c}_i\neq 0}}^n (\mathfrak{c}_i-1) \text{ and }
\mathfrak{s}=\sum_{\substack{i=1 \\ \mathfrak{s}_i\neq 0}}^n (\mathfrak{s}_i-1).$$

Let $G_S$ be the graph with vertex set $S$ and edges as follows. Two vertices in the same row $(r,c_1,s_1), (r,c_2,s_2)\in S$, with $c_1<c_2$, are adjacent if and only if there does not exist $(r,c_3,s_3)\in S$ with $c_1<c_3<c_2$. Two vertices in the same column $(r_1,c,s_1), (r_2,c,s_2)\in S$, with $r_1<r_2$, are adjacent if and only if there does not exists $(r_3,c,s_3)\in S$ with $r_1<r_3<r_2$.  Two vertices with the same symbol $(r_1,c_1,s), (r_2,c_2,s)\in S$, with $r_1<r_2$, are adjacent if and only if there does not exist $(r_3,c_3,s)\in S$ with $r_1<r_3<r_2$. By construction, $|E(G_S)|=\mathfrak{r + c + s}$.
Since $S$ is a 1TDS, every vertex of $S$ has at least one neighbour in $S$, and hence, every such vertex corresponds to an entry in $L$ that shares a common row or column or symbol with at least one other entry corresponding to a vertex in $S$. This implies that $G_S$ has no isolated vertex and hence
$$|E(G_S)|=\mathfrak{r + c + s}\geq \frac{|S|}{2} =\frac{\gamma_{\times 1,t}(\Gamma(L))}{2}.$$
%
%
%
By Theorem~\ref{theo:ktdsupperbound},  $|S|=\gamma_{\times 1,t}(\Gamma(L)) \leq  n-1$.
Hence, $L$ has at least one row and one column whose cells do not correspond to any entry of $S$.

Let $c_0$ denote the rightmost column of $L$ whose entries do not correspond to any element of $S$ and, similarly, $r_0$ the bottommost row of $L$ whose entries do not correspond to any element of $S$.
In $c_0$, there are $n-\gamma_{\times 1,t}(\Gamma(L))+\mathfrak{r}$ entries that do not share a common row or column with entries corresponding to vertices in $S$. Likewise, in $r_0$ there are $n-\gamma_{\times 1,t}(\Gamma(L))+\mathfrak{c}$ entries that do not share a common row or column with entries corresponding to vertices in $S$. Note there is one entry which is counted twice, i.e. the entry $(r_0,c_0,s)$ shared by the column $c_0$ and the row $r_0$.
This implies that the total number of entries in $c_0$ and $r_0$ that do not share a common row or column with entries corresponding to vertices in $S$ can be expressed as
$$(n-\gamma_{\times 1,t}(\Gamma(L))+\mathfrak{r})+(n-\gamma_{\times 1,t}(\Gamma(L))+\mathfrak{c})-1.$$
Since $S$ is a 1TDS all these vertices corresponding to entries in $c_0$ and $r_0$ are dominated by elements of $S$ that share the same symbol.
Note that any element of $S$ share the same symbol with at most two vertices corresponding to entries in all $c_0$ and $r_0$.
It follows that
$$(n-\gamma_{\times 1,t}(\Gamma(L))+\mathfrak{r})+(n-\gamma_{\times 1,t}(\Gamma(L))+\mathfrak{c})-1\leq 2(\gamma_{\times 1,t}(\Gamma(L))-\mathfrak{s})$$
and hence that
$$2n -1 + (\mathfrak{r} + \mathfrak{c} + 2\mathfrak{s}) \leq  4\gamma_{\times 1,t}(\Gamma(L)).$$
However, since $\mathfrak{r + c + s}\geq \frac{\gamma_{\times 1,t}(\Gamma(L))}{2}$ and $\mathfrak{s}\geq0$, we have
$$(2n -1)+ \frac{\gamma_{\times 1,t}(\Gamma(L))}{2} \leq 4\gamma_{\times 1,t}(\Gamma(L))$$
or equivalently
$$\frac{4n-2}{7} \leq \gamma_{\times 1,t}(\Gamma(L)).$$
\end{proof}

\begin{figure}[htp!]
\centering
\begin{tikzpicture}
\matrix[square matrix,nodes={draw,
      minimum height=11pt,
      anchor=center,
      text width=11pt,
      align=center,
      inner sep=0pt
    },]{
 1 & |[draw, circle]|2 & 3 & |[draw, circle]| 4 & 5 & 6 & 7 & 8 & 9 &10\\
 2 & 3 & 4 & 5 & |[draw, circle]| 1 & 7 & 8 & 9 & 10 & 6 \\
 3 & 4 & |[draw, circle]| 5 & 1 & 2 & 8 & 9 & 10 & 6 & 7 \\
 4 & 5 & 1 & 2 & |[draw, circle]| 3 & 9 & 10 & 6 & 7 & 8 \\
 |[draw, circle]| 5 & 1 & 2 & 3 & 4 & 10 & 6 & 7 & 8 & 9 \\
 6 & 7 & 8 & 9 & 10 & 1 & 2 & 3 & 4 & 5 \\  
 7 & 8 & 9 & 10 & 6 & 2 & 3 & 4 & 5 &1 \\
 8 & 9 & 10 & 6 & 7 & 3 & 4 & 5 & 1 & 2 \\
 9 & 10 & 6 & 7 & 8 & 4 & 5 & 1 & 2 & 3 \\
 10 & 6 & 7 & 8 & 9 & 5 & 1 & 2 & 3 & 4 \\
};
\end{tikzpicture}
\caption{A latin square of order $10$.}\label{fig:latinsqu10}
\end{figure}

\begin{Example} 
Consider $L$ the latin square of Figure~\ref{fig:latinsqu10}. If we consider $S=\{(1,2,2),(1,4,4),(2,5,1),(3,3,5),(4,5,3),(5,1,5)\}$, then we have that $|S|=6$ and it is a $1$TDS. On the other hand, by Theorem~\ref{theo:gamma1lowerbound4n7}, $\gamma_{\times 1,t}(\Gamma(L))\ge6$. This implies that $\gamma_{\times 1,t}(\Gamma(L))=6$.
\end{Example}

\section{$q$-step latin squares}

There exist several known classes of latin squares. Between them we recall the definition of the so called $q$-step latin squares. 

\begin{Definition}
A latin square $L$ of order $mq$ is said to be of \emph{$q$-step type} if it can be represented by a matrix of $q\times q$ blocks $A_{ij}$ as follows

\begin{center}
$$L=\begin{pmatrix}
A_{11} & A_{12} & \cdots & A_{1m} \\
A_{21} & A_{22} & \cdots & A_{2m} \\
\vdots   &  \vdots  & \cdots  &  \vdots\\
A_{m1} & A_{m2} & \cdots & A_{mm} \\
\end{pmatrix}$$
\end{center}
where each block $A_{ij}$ is a latin subsquare of order $q$ and two blocks $A_{ij}$ and $A_{i'j'}$ contain
the same symbols if and only if $i+j\equiv i'+j'$ (mod $m$).
\end{Definition}

\begin{Remark}\label{1-step}
Every cyclic latin square is a latin square of $1$-step type.
\end{Remark}

\begin{figure}[htp!]
\centering
\begin{tikzpicture}
\matrix[square matrix,nodes={draw,
      minimum height=11pt,
      anchor=center,
      text width=11pt,
      align=center,
      inner sep=0pt
    },]{
 1 & 2 & 3 & 4 & 5 & 6\\
 2 & 1 & 4 & 3 & 6 & 5 \\
 3 & 4 & 5 & 6 & 1& 2 \\
 4 & 3 & 6 & 5 & 2 & 1 \\
 5 & 6 & 1 & 2 & 3 & 4 \\
 6 & 5 & 2 & 1 & 4 & 3 \\  
};
\end{tikzpicture}
\caption{A $2$-step latin square of order $6$.}\label{fig:latinsqu3}
\end{figure}

By Theorem~\ref{theo:ktdsupperbound}, we know that $\gamma_{\times 1,t}(\Gamma(L)) \leq n-1$. However, in the case of $q$-step latin squares, we can describe a smaller upper bound.
\begin{Theorem}\label{theo:qstep1tdsupperbound}
Let $L$ be a $q$-step latin square of order $n=mq\ge3$. If $m\ge q+1$, then $\gamma_{\times 1,t}(\Gamma(L)) \leq n-q$.
If $m\le q$, then $\gamma_{\times 1,t}(\Gamma(L)) \leq n-m+1$.
\end{Theorem} 
\begin{proof}
If $q=1$, this is a consequence of Theorem~\ref{theo:ktdsupperbound}. We can suppose $q\ge2$.

First assume that $m\ge q+1$ and consider $S_1=\{(r,c,s)~|~1\le r\le q, (r-1)q+1\le c\le rq, 1\le s\le n\}$ and $S_2=\{(1,c,s)~|~q^2+1\le c\le n-q, 1\le s\le n\}$. Then $S=S_1\cup S_2$. In other words if 
$$L=\begin{pmatrix}
A_{11} & A_{12} & \cdots & A_{1m} \\
A_{21} & A_{22} & \cdots & A_{2m} \\
\vdots   &  \vdots  & \cdots  &  \vdots\\
A_{m1} & A_{m2} & \cdots & A_{mm} \\
\end{pmatrix}$$
$S_1$ consists of the first row of the block $A_{11}$, the second row of the block $A_{12}$, and so on until the last row of the block $A_{1q}$. This implies that $|S_1|=q^2$. Similarly, $S_2$ is the first row of all the blocks $A_{1j}$ with $q+1\le j\le m-1$, and $|S_2|=(m-1)q-q^2$.
This implies that $|S|=(m-1)q=n-q$. By construction $S$ is a $1$TDS. In fact, if $v\in S$, then it has at least $q-1$ neighbours in $S$, if $v\notin S$ then it has at least one neighbour in $S$. Hence $\gamma_{\times 1,t}(\Gamma(L)) \leq n-q$.

Assume now that $m\le q$ and consider $S_1=\{(r,c,s)~|~1\le r\le m-1, (r-1)q+1\le c\le rq, 1\le s\le n\}$ and $S_2=\{(r,(m-1)q,s)~|~m\le r \le q, 1\le s\le n\}$. Then $S=S_1\cup S_2$. In other words, $S_1$ consists of the first row of the block $A_{11}$, the second row of the block $A_{12}$, and so on until the $(m-1)$-th row of the block $A_{1(m-1)}$. This implies that $|S_1|=(m-1)q$. Similarly, $S_2$ is the bottom part of the last column of the block $A_{1(m-1)}$, and $|S_2|=q-(m-1)$.
This implies that $|S|=mq-m+1=n-m+1$. By construction $S$ is a $1$TDS. In fact, if $v\in S_1$, then it has $q-1$ neighbours in $S$, if $v\in S_2$, then it has $q-(m-1)\ge1$ neighbours in $S$,  if $v\notin S$ then it has at least one neighbour in $S$. Hence $\gamma_{\times 1,t}(\Gamma(L)) \leq n-m+1$.
\end{proof}

If we use the technique described in the proof of Theorem~\ref{theo:qstep1tdsupperbound}, we can easily construct $1$TDS for $q$-step latin squares.

\begin{Example} 
Consider $L$ the $2$-step latin square of Figure~\ref{fig:latinsqu3}. In this case, $q=2$ and $m=3$. Consider $S=\{(1,1,1),(1,2,2),(2,3,4),(2,4,3)\}$, then we have that $|S|=6-2=4$ and it is a $1$TDS as described in the proof of Theorem~\ref{theo:qstep1tdsupperbound}. On the other hand, by Theorem~\ref{theo:gamma1lowerbound4n7}, $\gamma_{\times 1,t}(\Gamma(L))\ge4$. This implies that $\gamma_{\times 1,t}(\Gamma(L))=4$.
\end{Example}

\begin{Example}
Consider $L$ the $3$-step latin square of Figure~\ref{fig:latinsqu3step9}. In this case, $q=m=3$. If we consider $S=\{(1,1,1),(1,2,2),(1,3,3),(2,4,5),$ $(2,5,6),$ $(2,6,4),(3,6,5)\}$, we have that $|S|=9-3+1=7$ and it is a $1$TDS as described in the proof of Theorem~\ref{theo:qstep1tdsupperbound}.
\end{Example}

\begin{figure}[htp!]
\centering
\begin{tikzpicture}
\matrix[square matrix,nodes={draw,
      minimum height=11pt,
      anchor=center,
      text width=11pt,
      align=center,
      inner sep=0pt
    },]{
 |[draw, circle]|1 & |[draw, circle]|2 & |[draw, circle]|3 & 4 & 5 & 6 & 7 & 8 & 9 \\
 2 & 3 & 1 & |[draw, circle]|5 & |[draw, circle]|6 & |[draw, circle]|4 & 8 & 9 &7 \\
 3 & 1 & 2 & 6 & 4 & |[draw, circle]|5 & 9 & 7& 8 \\
 4 & 5 & 6 & 7 & 8 & 9 & 1 & 2 & 3 \\
 5 & 6 & 4 & 8 & 9 & 7 & 2 & 3 & 1 \\
 6 & 4 & 5 & 9 & 7 & 8 & 3 & 1 & 2 \\  
 7 & 8 & 9 & 1 & 2 & 3 & 4 & 5 & 6 \\
 8 & 9 & 7 & 2 & 3 & 1 & 5 & 6 & 4 \\
 9 & 7 & 8 & 3 & 1 & 2 & 6 & 4 & 5\\
};
\end{tikzpicture}
\caption{A $3$-step latin square of order $9$.}\label{fig:latinsqu3step9}
\end{figure}

%


\section{Dominating set}
Similarly to the case of $1$TDS and $2$TDS, if we consider latin squares of small order, we can easily compute their domination number.

\begin{Lemma}\label{lemma:dominationorder5} Let $L$ be a latin square of order $2\le n\le5$. Then 
$$\gamma(\Gamma(L)) = 
\begin{cases}
n-1 & \text{ if $2\le n\le4$ }\\
3 & \text{ if $n=5$.}
\end{cases}
$$
\end{Lemma}

Putting together \eqref{eq:gammasmallgammatot} and Theorem~\ref{theo:ktdsupperbound}, we have that $\gamma(\Gamma(L)) \leq n-1$. However, 
we can describe a smaller upper bound.
\begin{Theorem}\label{theo:dominationupperbound} Let $L$ be a latin square of order $n\ge5$. Then 
$$\gamma(\Gamma(L)) \le 
\begin{cases}
  n-2 & \text{ if $5\le n\le 21$} \\
  n-(\lceil \frac{\sqrt{2n+7}-1}{2} \rceil-1)  & \text{ if $n\ge 22$.}  
  \end{cases}
$$
\end{Theorem}
\begin{proof} Assume first that $5\le n\le 21$. By Lemma~\ref{lemma:dominationorder5}, we can assume that $6\le n\le 21$. 
Without loss of generality, we can assume that the bottom right $2\times2$ submatrix of $L$ is the square in Figure~\ref{fig:specialsquareorder2}, for some symbols $a,b\ne1$.
\begin{figure}[htp!]
\centering
\begin{tikzpicture}
\matrix[square matrix,nodes={draw,
      minimum height=11pt,
      anchor=center,
      text width=11pt,
      align=center,
      inner sep=0pt
    },]{
 a & 1 \\
 1 & b \\ 
};
\end{tikzpicture}
\caption{Square of order $2$.}\label{fig:specialsquareorder2}
\end{figure}

Assume that $a= b$ and let $L'$ be the $(n-2)\times(n-2)$ submatrix of $L$ obtained from $L$ by deleting the last two rows and the last two columns (notice that $L'$ is not a latin square in general). In $L'$ there are exactly $n-2$ entries with symbol $a$ and $n-2$ entries with symbol $1$. Let $v_1=(r_1,c_1,a)$ and $v_2=(r_2,c_2,1)$ be such entries in $L'$ with $r_1\ne r_2$ and $c_1\ne c_2$. For all $i\in\{3,\dots,n-2\}$, let $v_i=(r_i,c_i,s_i)$ be an entry in $L'$ with $r_i\ne r_j$, $c_i\ne c_j$ for all $j\in\{1,\dots,i-1\}$, and $s_i\in\{1,\dots,n\}$. Consider $S=\{v_1,v_2,v_3,\dots,v_{n-2}\}$. By construction $S$ dominates all the vertices of $\Gamma(L)$ that correspond to entries in the first $n-2$ rows of $L$ by row, to the entries of $L$ in the first $n-2$ columns by column and to the bottom right $2\times2$ submatrix of $L$ by symbol. This implies that $S$ is a dominating set for $\Gamma(L)$, and hence $\gamma(\Gamma(L)) \le |S|= n-2$.

Assume that $a\ne b$ and let $L'$ be the $(n-2)\times(n-2)$ submatrix of $L$ obtained from $L$ by deleting the last two rows and the last two columns (notice that $L'$ is not a latin square in general). In $L'$ there are exactly $n-3$ entries with symbol $a$. Let $v_1=(r_1,c_1,a)$ be one of such entries in $L'$. Since $n\ge6$, we can assume that $v_1$ is chosen in such way that if in $L'$ we have entries $(r_1,c',1)$ and $(r',c_1,1)$, then in $L$ we have at most one of the entries $(n-1,c',b)$ and $(r',n-1,b)$. In $L'$ there are at least $n-5$ entries with symbol $b$ that are not in the row $r_1$ or in the column $c_1$. Let $v_2=(r_2,c_2,b)$ be one such entry in $L'$ such that the entry $(r_1,c_2,1)$ or $(r_2,c_1,1)$ is an entry in $L'$. Notice that such entry $v_2$ always exists by the choice of $v_1$. By construction, in $L'$ there are at least $n-5$ entries with symbol $1$ that are not in the rows $r_1,r_2$ or in the columns $c_1,c_2$. Since $n\ge6$, we can consider $v_3=(r_3,c_3,1)$ be one of such entries in $L'$.
For all $i\in\{4,\dots,n-2\}$, let $v_i=(r_i,c_i,s_i)$ be an entry in $L'$ with $r_i\ne r_j$, $c_i\ne c_j$ for all $j\in\{1,\dots,i-1\}$, and $s_i\in\{1,\dots,n\}$. Consider $S=\{v_1,v_2,v_3,v_4,\dots,v_{n-2}\}$. By construction $S$ dominates all the vertices of $\Gamma(L)$ that correspond to entries in the first $n-2$ rows of $L$ by row, to the entries of $L$ in the first $n-2$ columns by column and to the bottom right $2\times2$ submatrix of $L$ by symbol. This implies that $S$ is a dominating set for $\Gamma(L)$, and hence $\gamma(\Gamma(L)) \le |S|= n-2$.

Assume now that $n\ge22$. The previous argument can be easily generalized. Consider $2\le k\le \lceil \frac{n}{2}\rceil$, let $L'$ be the $(n-k)\times(n-k)$ submatrix of $L$ obtained from $L$ by deleting the last $k$ rows and the last $k$ columns, and let $L''$ be the $k\times k$ submatrix of $L$ obtained from $L$ by deleting the first $n-k$ rows and the first $n-k$ columns (notice that $L'$ and $L''$ are not latin squares in general).  We construct $S$ a diagonal of $L'$ (i.e., $n-k$ entries in distinct rows and columns) containing all the symbols that appear in the entries of $L''$ as follows.  We choose entries in $L'$ which match the symbols that appear in the entries of $L''$ one by one. Consider $(r,c,s)$ an entry of $L''$. Then there are at most $2k-1$ entries in $L$ with symbol $s$ that are not in $L'$. On the other hand, in $L''$ appear at most $k^2$ symbols. This implies that, when processing symbol $s$ in $L''$, we can choose any entry of $L'$ with symbol $s$ other than the at most $2k-1$ copies of $s$ outside of $L'$, and the at most $ 2(k^2-1)$ copies of $s$ inside of $L'$ which occur in the same row or column as a previously chosen entry.  Thus, this works if $n > 2k-1 + 2(k^2-1)$, or equivalently $k < (\sqrt{2n+7}-1)/2$ which occurs when $k\le \lceil(\sqrt{2n+7}-1)/2\rceil-1$. Now that all the symbols appearing in $L''$ have been taken care, we can complete $S$ with enough entries of $L'$ belonging to rows and columns without a chosen entry. In this way, $S$ is a dominating set for $\Gamma(L)$ with $|S|=n-k$.
\end{proof}

\begin{Remark} Notice that the first part of the proof of Theorem~\ref{theo:dominationupperbound} works for any $n\ge5$. However, starting from $n=22$, where $n-2=20>19=n-(\lceil \frac{\sqrt{2n+7}-1}{2} \rceil-1)$, $n-(\lceil \frac{\sqrt{2n+7}-1}{2} \rceil-1)$ is a smaller upper bound than $n-2$.
\end{Remark}

In the case of $1$-step latin squares, we can describe an even smaller upper bound than the one of Theorem~\ref{theo:dominationupperbound}.

\begin{Theorem}\label{theo:dominsetupperbound1step}
Let $L$ be a $1$-step (i.e. cyclic) latin square of order $n=3f+g$ where $f\ge1$ and $0\leq g <3$. Then $\gamma(\Gamma(L))\leq 2f+g.$
\end{Theorem}
\begin{proof}
First assume that $g=0$. Split $L$ into $9$ regions of equal size like a sudoku. Label the bottom regions of $L$ by I-III from left to right, the middle regions by IV-VI, and the top regions by VII-IX. 
Construct $S$ by taking the vertices corresponding to the entries on
the first upper diagonal in region I together with the vertex corresponding to the entry in the bottom left corner of region I, and the vertices corresponding to entries along the main
diagonal of region IX. 
In this construction it can be seen that $|S|=2f$ and that there is exactly one element of $S$ in
each column and row of regions I and IX.
The set $S$ constructed in this way is a dominating set. To see this, one can simply note that the entries in region I-III and regions VII-IX are all dominated row-wise by the elements in regions I and IX respectively. Regions IV and VI are dominated column-wise by the elements in regions I and IX respectively. This leaves region V which is diagonally dominated by the elements in regions I and IX. This implies that $\gamma(\Gamma(L))\leq 2f$.

Assume now that $1\le g\le 2$. In this case, consider $L'$ the submatrix of $L$ obtained by deleting the last $g$ rows and $g$ columns. By construction $L'$ is a $3f\times 3f$ submatrix of $L$. Similarly to the case $g=0$, we can construct a dominating set $S'$ for $\Gamma(L')$ such that $|S'|=2f$. To obtain $S$ a dominating set for $\Gamma(L)$, it is enough to add to $S'$ the $g$ vertices corresponding to entries in the last $g$ rows on the main diagonal of $L$. $S$ is clearly a dominating set for $\Gamma(L)$ and $|S|=2f+g$. This implies that $\gamma(\Gamma(L))\leq 2f+g$.
\end{proof}

If we use the technique described in the proof of Theorem~\ref{theo:dominsetupperbound1step}, we can easily construct dominating set for cyclic latin squares.

\begin{figure}[htp!]
\centering
\begin{minipage}{0.4\textwidth}
\centering
\begin{tikzpicture}
\matrix[square matrix,nodes={draw,
      minimum height=11pt,
      anchor=center,
      text width=11pt,
      align=center,
      inner sep=0pt
    },]{
 1 & 2 & 3 & 4 & |[draw, circle]|5 & 6\\
 2 & 3 & 4 & 5 & 6 & |[draw, circle]|1 \\
 3 & 4 & 5 & 6 & 1& 2 \\
 4 & 5 & 6 & 1 & 2 & 3 \\
 5 &|[draw, circle]|6 & 1 & 2 & 3 & 4 \\
 |[draw, circle]|6 & 1 & 2 & 3 & 4 & 5 \\  
};
\end{tikzpicture}
\end{minipage}
\begin{minipage}{0.4\textwidth}
\centering
\begin{tikzpicture}
\matrix[square matrix,nodes={draw,
      minimum height=11pt,
      anchor=center,
      text width=11pt,
      align=center,
      inner sep=0pt
    },]{
 1 & 2 & 3 & 4 & 5 & 6 & |[draw, circle]|7 & 8 & 9 \\
 2 & 3 & 4 & 5 & 6 & 7 & 8 & |[draw, circle]|9 &1 \\
 3 & 4 & 5 & 6 & 7 & 8 & 9 & 1& |[draw, circle]|2 \\
 4 & 5 & 6 & 7 & 8 & 9 & 1 & 2 & 3 \\
 5 & 6 & 7 & 8 & 9 & 1 & 2 & 3 & 4 \\
 6 & 7 & 8 & 9 & 1 & 2 & 3 & 4 & 5 \\  
 7 & |[draw, circle]|8 & 9 & 1 & 2 & 3 & 4 & 5 & 6 \\
 8 & 9 & |[draw, circle]|1 & 2 & 3 & 4 & 5 & 6 & 7 \\
 |[draw, circle]|9 & 1 & 2 & 3 & 4 & 5 & 6 & 7 & 8\\
};
\end{tikzpicture}
\end{minipage}
\caption{A cyclic latin square of order $6$ and one of order $9$.}\label{fig:1steporder6}
\end{figure}

\begin{figure}[htp!]
\centering
\begin{tikzpicture}
\matrix[square matrix,nodes={draw,
      minimum height=11pt,
      anchor=center,
      text width=11pt,
      align=center,
      inner sep=0pt
    },]{
 1 & 2 & 3 & 4 & |[draw, circle]|5 & 6 &7\\
 2 & 3 & 4 & 5 & 6 & |[draw, circle]|7 & 1 \\
 3 & 4 & 5 & 6 & 7 & 1 & 2 \\
 4 & 5 & 6 & 7 & 1 & 2 & 3 \\
 5 & |[draw, circle]|6 & 7 & 1 & 2 & 3 & 4 \\
 |[draw, circle]|6 & 7 & 1 & 2 & 3 & 4 & 5 \\  
7 & 1 & 2 & 3 & 4 & 5 & |[draw, circle]|6  \\  
};
\end{tikzpicture}
\caption{A cyclic latin square of order $7$.}\label{fig:1steporder7}
\end{figure}

\begin{figure}[htp!]
\centering
\begin{tikzpicture}
\matrix[square matrix,nodes={draw,
      minimum height=11pt,
      anchor=center,
      text width=11pt,
      align=center,
      inner sep=0pt
    },]{
 1 & 2 & 3 & 4 & |[draw, circle]|5 & 6 &7 & 8\\
 2 & 3 & 4 & 5 & 6 & |[draw, circle]|7 & 8 & 1 \\
 3 & 4 & 5 & 6 & 7 & 8 & 1 & 2 \\
 4 & 5 & 6 & 7 & 8 & 1 & 2 & 3 \\
 5 & |[draw, circle]|6 & 7 & 8& 1 & 2 & 3 & 4 \\
 |[draw, circle]|6 & 7 & 8 & 1 & 2 & 3 & 4 & 5 \\  
7 & 8 & 1 & 2 & 3 & 4 & |[draw, circle]|5 & 6  \\  
8 & 1 & 2 & 3 & 4 & 5 & 6 & |[draw, circle]|7  \\  
};
\end{tikzpicture}
\caption{A cyclic latin square of order $8$.}\label{fig:1steporder8}
\end{figure}

\begin{Example}
Consider $L_1$ the cyclic latin square of order $6$ of Figure~\ref{fig:1steporder6}. If we consider $S_1=\{(1,5,5),(2,6,1),(5,2,6),(6,1,6)\}$, then $|S_1|=4$ and it is a dominating set for $\Gamma(L_1)$.

Consider $L_2$ the cyclic latin square of order $9$ of Figure~\ref{fig:1steporder6}. If we consider $S_2=\{(1,7,7),(2,8,9),(3,9,2),(7,2,8),(8,3,1),(9,1,9)\}$, then $|S_2|=6$ and it is a dominating set for $\Gamma(L_2)$.

Consider $L_3$ the cyclic latin square of Figure~\ref{fig:1steporder7}. If we consider $S_3=\{(1,5,5),(2,6,7),(5,2,6),(6,1,6),(7,7,6)\}$, then $|S_3|=5$ and it is a dominating set for $\Gamma(L_3)$.

Consider $L_4$ the cyclic latin square of Figure~\ref{fig:1steporder8}. If we consider $S_4=\{(1,5,5),(2,6,7),(5,2,6),(6,1,6),(7,7,5),(8,8,7)\}$, then $|S_4|=6$ and it is a dominating set for $\Gamma(L_4)$.

\end{Example}

We can now describe a lower bound on $\gamma(\Gamma(L))$.

\begin{Theorem}\label{theo:lowbanddominsetnumb} Let $L$ be a latin square of order $n\ge5$. Then $$\gamma(\Gamma(L))\ge\lceil{\frac{n}{2}}\rceil.$$
\end{Theorem}
\begin{proof} 
Assume there exists $S$ a dominating set of $\Gamma(L)$ such that $|S|<\frac{n}{2}$. This implies that there are more than $\frac{n}{2}$ rows and $\frac{n}{2}$ columns of $L$ whose entries do not correspond to vertices in $S$. Let $L'$ be the $p\times q$ submatrix of $L$ obtained by deleting the rows and columns that contain cells corresponding to the entries of $S$, where $p,q\ge n-|S|$ (notice that $L'$ is not a latin square in general). 
Since $|S|<\frac{n}{2}$, then there exist a symbol $s_0$ that does not appear in any entry in $S$. Moreover, since $n-|S|>\frac{n}{2}$, then every symbol must occur in $L'$. In particular, there exists one entry of the form $(r,c,s_0)\in E(L')$. However, $(r,c,s_0)$ has no neighbours in $S$, but this is an absurd. Hence, $\gamma(\Gamma(L))\ge\lceil\frac{n}{2}\rceil.$
%
%
\end{proof}

\begin{figure}[htp!]
\centering
\begin{tikzpicture}
\matrix[square matrix,nodes={draw,
      minimum height=11pt,
      anchor=center,
      text width=11pt,
      align=center,
      inner sep=0pt
    },]{
 1 & 2 & 3 & 4 & 5 & 6 & 7 \\
 2 & 3 & 1 & 5 & 6 & 7 & 4 \\
 3 & 1 & 2 & 6 & 7 & 4 & 5 \\
 4 & 5 & 6 & 7 & 1 & 2 & |[draw, circle]|3 \\
 5 & 6 & 7 & 1 & |[draw, circle]|4 & 3 & 4 \\
 6 & 7 & 4 & |[draw, circle]|2 & 3 & 5 & 1 \\
 7 & 4 & 5 & 3 & 2 & |[draw, circle]|1 & 6 \\  
};
\end{tikzpicture}
\caption{A latin square of order $7$ with $\gamma(\Gamma(L))=4$. }\label{fig:latinsqu7gamma4}
\end{figure}

\begin{Remark} By \cite[Theorem 12]{BMSW}, the lower bound described in Theorem~\ref{theo:lowbanddominsetnumb} is tight in all cases.
\end{Remark}

\begin{Example} Consider $L$ the latin square of Figure~\ref{fig:latinsqu7gamma4}. If we consider $S=\{(4,7,3),(5,5,4),(6,4,2),(7,6,1)\}$, then $|S|=4$ and it is a dominating set for $\Gamma(L)$. By Theorem~\ref{theo:lowbanddominsetnumb}, $\gamma(\Gamma(L))\ge\frac{7}{2}$. However, the domination number is an integer, and hence $\gamma(\Gamma(L))=4$.
\end{Example}

\begin{Theorem} Let $L$ be a $1$-step (i.e. cyclic) latin square of even order $n$. If $\frac{n}{2}$ is odd, then $\gamma(\Gamma(L))= \frac{n}{2}$. If $\frac{n}{2}$ is even, then $\gamma(\Gamma(L))\le \frac{n}{2}+1$.
\end{Theorem}
\begin{proof} Let $L'$ be the $\frac{n}{2}\times \frac{n}{2}$ submatrix of $L$ obtained from $L$ by deleting every second row and every second column.
Notice that $L'$ is isotopic to a cyclic Latin square. By \cite{WAN2004}, $L'$ has a transversal if and only if $\frac{n}{2}$ is odd. This implies that if $\frac{n}{2}$ is odd, we can consider $S$ a transversal of $L'$. However such $S$ forms a dominating set of $\Gamma(L)$. By Theorem~\ref{theo:lowbanddominsetnumb}, $\gamma(\Gamma(L))= \frac{n}{2}$.

On the other hand, if $\frac{n}{2}$ is even, we can construct $S'$ a partial transversal of $L'$ of cardinality $(\frac{n}{2})-1$. This implies that adding two entries to $S'$ we can obtain $S$ a dominating set of $\Gamma(L)$ with $|S|=(\frac{n}{2})+1$. Hence $\gamma(\Gamma(L))\le (\frac{n}{2})+1$.
\end{proof}

In the proof of \cite[Theorem 13]{BMSW}, it is shown that if you take a dominating set of size $n-d$, then it must
miss at least $d$ rows, at least $d$ columns and at least $d$ symbols. If you look at the submatrix formed by the rows that are missed and
the columns that are missed, then it cannot contain any of the symbols that are missed. This fact allows us to rewrite \cite[Theorem 13]{BMSW}, and obtain the following result.

\begin{Theorem} Fix $\varepsilon>0$. Almost all latin squares will not have a dominating set of size less than $n-O(n^{\frac{2}{3}+\varepsilon})$.
\end{Theorem}

\paragraph{\textbf{Acknowledgements}} During the preparation of this article the third author was supported by JSPS Grant-in-Aid for Early-Career Scientists (19K14493).



\begin{thebibliography}{20}

\bibitem{BMSW} D. Best, T. Marbach, R.~J. Stones and I.~M. Wanless, Covers and partial transversals of Latin squares, {\em Des. Codes Cryptogr.} {\bf 87} (2019), 1109-1136.

\bibitem{bondy2008graph} J.~A. Bondy and U.~S.~R. Murty, {\em Graph theory},  {\em Graduate texts in mathematics}. vol. 244, Springer Science and Media, 2008.

\bibitem{bose1963} R.~C. Bose, Strongly regular graphs, partial geometries and partially balanced designs, {\em Pacific J. Math.} {\bf 13} (1963) 389-419.

\bibitem{CombinDes} C.~J.Colbourn and J.~H. Dinitz, {\em Handbook of Combinatorial Designs}, {\em Second Edition (Discrete Mathematics and Its Applications)}, Chapman and Hall/CRC, 2006.

\bibitem{EMW19} A.~B. Evans, A. Mammoliti and I.~M. Wanless, Latin squares with maximal partial transversals of many lengths, {\em J. Combin. Theory Ser. A}, {\bf 180} (2021), 105403.

\bibitem{GR01} C. Godsil and G. Royle, {\em Algebraic Graph Theory}, {\em Graduate Texts in Mathematics}, vol. 207, Springer-Verlag, New York, 2001.

\bibitem{HH3}  F. Harary and T. W. Haynes, Double domination in graphs, {\em Ars Combin.} {\bf 55} (2000), 201-213.

\bibitem{HHS5}  T. W. Haynes, S. T. Hedetniemi and P. J. Slater, {\em Fundamentals of Domination in Graphs}, Marcel Dekker, New York, 1998.

\bibitem{HHS6}  T. W. Haynes, S. T. Hedetniemi and P. J. Slater, {\em Domination in Graphs: Advanced Topics}, Marcel Dekker, New York, 1998.

\bibitem{HK8}  M. A. Henning and A. P. Kazemi, $k$-tuple total domination in graphs, {\em Discrete Appl. Math.} {\bf 158} (2010), 1006-1011.

\bibitem{KP}  A. P. Kazemi and B. Pahlavsay, Quasi-transversal in latin squares, {\em arXiv:1808.05213}.

\bibitem{KPPT2008} F. Kazemnejad, B. Pahlavsay, E. Palezzato and M. Torielli, Domination number of middle graphs. {\em arXiv:2008.02975}.

\bibitem{PPT2018} B. Pahlavsay, E. Palezzato and M. Torielli, $3$-tuple total domination number of rook's graphs. {\em To appear in Discussiones Mathematicae Graph Theory} (2021).

\bibitem{WAN2004} I.~M. Wanless, Diagonally cyclic latin squares. {\em European J. Combin.} {\bf 25(3)} (2004), 393-413.


\end{thebibliography}
\end{document}